\documentclass[12pt]{amsart}
\usepackage[active]{srcltx}
\usepackage{a4wide}
\usepackage{amsthm,amsfonts,amsmath,mathrsfs,amssymb}
\usepackage{dsfont}
\usepackage[T1]{fontenc}
\usepackage[utf8]{inputenc}
\usepackage{fourier}
\newcommand{\D}{\mathbb{D}}

\newcommand{\T}{\mathbb{T}}

\newcommand{\Hp}{{\mathscr{H}}^q}

\newcommand{\Real}{\operatorname{Re}}

\newcommand{\norm}[1]{\left\|#1\right\|}
\newtheorem{theorem}{Theorem}

\newtheorem{lemma}{Lemma}

\begin{document}

\title[An inequality of Hardy--Littlewood type  for Dirichlet polynomials]{An inequality of Hardy--Littlewood type for Dirichlet polynomials }

\begin{abstract}
The $L^q$ norm of a Dirichlet polynomial $F(s)=\sum_{n=1}^{N} a_n n^{-s}$ is defined as
\[\| F\|_q:=\left(\lim_{T\to\infty}\frac{1}{T}\int_{0}^T |F(it)|^qdt\right)^{1/q}\]
for $0<q<\infty$. It is shown that
\[ \left(\sum_{n=1}^{N} |a_n|^2|\mu(n)|[d(n)]^{\frac{\log q}{\log 2} -1}\right)^{1/2}\le \| F\|_q \]
when $0<q<2$; here $\mu$ is the M\"{o}bius function and $d$ the divisor function. This result is used to prove that the $L^q$ norm of $D_N(s):=\sum_{n=1}^{N} n^{-1/2-s}$ satisfies $\|D_N\|_q\gg (\log N)^{q/4}$ for $0<q<\infty$. By Helson's generalization of the M. Riesz theorem on the conjugation operator, the reverse inequality $\|D_N\|_q \ll (\log N)^{q/4}$ is  shown to be valid in the range $1<q<\infty$. Similar bounds are found for a fairly large class of Dirichlet series including, on one of Selberg's conjectures, the Selberg class of $L$-functions.
\end{abstract}



\author{Andriy Bondarenko}
\address{Department of Mathematical Analysis\\ Taras Shevchenko National University of Kyiv\\
Volody- myrska 64\\ 01033 Kyiv\\ Ukraine}
\address{Department of Mathematical Sciences \\ Norwegian University of Science and Technology \\ NO-7491 Trondheim \\ Norway}

\email{andriybond@gmail.com}

\author[Winston Heap]{Winston Heap}
\address{Department of Mathematical Sciences, Norwegian University of Science and Technology,
NO-7491 Trondheim, Norway} \email{winstonheap@gmail.com}

\author[Kristian Seip]{Kristian Seip}
\address{Department of Mathematical Sciences \\ Norwegian University of Science and Technology \\ NO-7491 Trondheim \\ Norway}
\email{seip@math.ntnu.no}
\thanks{Research supported by Grant 227768 of the Research Council of Norway. }


\maketitle

\section{Introduction}
This paper estimates $L^q$ norms of Dirichlet polynomials $F(s)=\sum_{n=1}^{N} a_n n^{-s}$,  defined as
\[\| F\|_q:=\left(\lim_{T\to\infty}\frac{1}{T}\int_{0}^T |F(it)|^qdt\right)^{1/q}\]
for $0<q<\infty$.  We will establish a basic inequality for this norm which has its origin in certain inequalities on the unit circle studied by Hardy and Littlewood and many other authors. We will use this inequality to obtain lower bounds for $L^q$ norms of partial sums of Dirichlet series whose coefficients are multiplicative arithmetic functions $a(n)$ satisfying one or both of the following two conditions:
\begin{itemize}
\item[(A)] There exist two constants $C>0$ and $\theta<1/4$ such that $|a(p^m)|\le C p^{\theta m}$ for all primes $p$ and positive integers $m$.
\item[(B)] There exists a constant $\theta<1/2$ such that $a(p)\ll p^{\theta}$ for all primes $p$ and, for every real number $r$, we have
\begin{equation}\label{normbelow} \prod_{p \le x} (1+2^r |a(p)|^2p^{-1}) \ll_r \sum_{n\le x} |\mu(n)| |a(n)|^2 n^{-1} [d(n)]^r, \end{equation}
\end{itemize}
where $\mu(n)$ is the M\"{o}bius function. We see that condition (A) is a mild growth condition, while the rationale for the more subtle condition (B) will become clear in the light of our basic inequality.  
A simple argument (see the remark after the proof of Lemma~\ref{square} below) shows that (B) implies that
\begin{equation}\label{conseq} \lambda_{a}(x):= \sum_{p\le x}\frac{|a(p)|^2}{p}\ll \log x.\end{equation}
Conversely, we will show in Section~\ref{ramanujan} that \eqref{normbelow} is indeed satisfied whenever
\begin{equation}\label{selberg} \lambda_{a}(x)=c \log\log x +O(1). \end{equation}
This means that  the constant sequence $a(n)\equiv 1$ satisfies (B). More generally, we note that, on one of Selberg's conjectures \cite{S2}, any function $a(n)$ representing the coefficients of an $L$-function in the Selberg class meets (B); condition (A) is  trivially satisfied by such an $a(n)$ in view of the definition of the Selberg class.

\begin{theorem}\label{thm1}
Suppose $a(n)$ is a multiplicative arithmetic function, and set
\[D_N(s)=\sum_{n=1}^Na(n)n^{-1/2-s}.\]
If $a(n)$ satisfies $\operatorname{(A)}$, then
\begin{equation} \label{above}\|D_N\|_q\ll_q \begin{cases} e^{q\lambda_{a}(N)/4}, & q>1 \\
\lambda_a(N) e^{\lambda_{a}(N)/4}, & q=1 \\
e^{\lambda_{a}(N)/4}, & 0<q<1.
\end{cases} \end{equation}
On the other hand, if $a(n)$ satisfies $\operatorname{(B)}$, then
\begin{equation}\label{below} \| D_N\|_q\gg_q e^{q\lambda_{a}(N)/4} \end{equation}
for every $0< q <\infty$. 
\end{theorem}


In the distinguished case when $a(n)\equiv 1$ and $q$ is an even integer, a precise asymptotic expression for $\|D_N\|_q$ is known from the work of Conrey and Gamburd \cite{CG}. We do not reach this level of precision, but we would like to stress that the point of our Theorem~\ref{thm1} is that we have dispensed with Hilbert space methods and found the right order of magnitude of the norm $\| D_N\|_q$ for a continuous range of $q$.

We note that our bounds are consistent with conjectures for the $2k$th moment of a primitive $L$-function from the Selberg class. Indeed, Conjecture 2.5.4 of Conrey et al. \cite{cfkrs} states that for such an $L$-function $L(s)$,
\[\frac{1}{T}\int_0^T|L(\tfrac{1}{2}+it)|^{2k}dt\sim c_L(k)(\log T)^{k^2},\,\,\,\,\,\,\,\,k\in\mathbb{N}\]
for some constant $c_L(k)$. Our sharp asymptotic estimate $\|D_N\|_q\asymp (\log N)^{q/4}$ when $a(n)\equiv 1$ in the range $1<q<\infty$ is in line with this conjecture, as is our lower bound \eqref{below} for all $q>0$. Radziwi{\l\l} and
Soundararajan \cite{RS} have verified that the $2k$th moment of the Riemann zeta function is bounded below by $C_k(\log T)^{k^2}$ for real $k> 1$; Heath-Brown \cite{HB} obtained earlier the same result for all rational values of $k$. On the Riemann Hypothesis, the latter bound is known for all $k>0$ \cite{HB}. Harper~\cite{Har1}, building on work of Soundararajan~\cite{Sound1}, showed recently that the upper bounds of  optimal order $(\log T)^{k^2}$ also hold conditionally for all $k>0$. Finally, in upcoming work, Radziwi{\l\l} and Soundararajan \cite{RadziwillSoundararajan2} will establish unconditionally a bound for the correct order of magnitude for all fractional moments with $0 < k < 4$.

We will see in Section~\ref{proofzeta} below that the estimate from above in Theorem~\ref{thm1} is a fairly easy consequence of Helson's generalization of the M. Riesz theorem on the conjugation operator \cite{H1}, applied to certain finite Euler products.
As to the bound from below, we start by recalling the following interesting lower bound found by Helson \cite{H} (see also \cite[Theorem 6.5.9]{QQ}):
\begin{equation} \label{helson} \left(\sum_{n=1}^{N} |a_n|^2 |[d(n)]^{-1}\right)^{1/2}\le \| F\|_1, \end{equation}
where $d(n)$ denotes the divisor function. This inequality shows the relevance of the expression on the right-hand side of \eqref{normbelow} when $r=-1$, apart from the appearance of the M\"{o}bius function. Before explaining the role of the factor $|\mu(n)|$, we need to take a closer look at \eqref{helson}. This bound was obtained by a suitable iteration of the
inequality
\begin{equation} \label{HL} \left(\sum_{n=0}^{\infty} |c_n|^2 (n+1)^{-1}\right)^{1/2}\le \| f\|_{H^1(\T)}, \end{equation}
valid for $f(z)=\sum_{n=0}^{\infty} c_n z^n$  in the Hardy space $H^{1}(\T)$ of the unit circle (see Section~\ref{proofembed} for the definition of the spaces $H^q(\T)$). The latter result goes back to Carleman \cite{C} and has later been rediscovered by several authors (see  e.~g. \cite{Ma, V}). Here it is essential that the norm on the right-hand side is computed with respect to normalized Lebesgue measure on $\T$ and that the inequality is contractive. A noncontractive version of \eqref{HL} follows from the better known inequality
\[ \sum_{n=0}^{\infty} |c_n| (n+1)^{-1}\le \pi \| f\|_{H^1(\T)} \]
of Hardy and Littlewood \cite{HL1}. See also Hardy and Littlewood's paper \cite{HL} which contains an elaborate study of similar inequalities.

For the proof of \eqref{below}, we will use an $L^q$ version of \eqref{helson}, and this is what led us to condition (B). To be more precise, we need to keep a suitable weighted $\ell^2$-norm on the left-hand side and replace the $L^1$ norm by the $L^q$ norm on the right-hand side of the inequality. Our basic inequality  is based on the observation that this becomes a  manageable task if we sum only over square-free numbers:

 \begin{theorem}\label{embed}
Let $F(s)=\sum_{n=1}^{N} a_n n^{-s}$ be a Dirichlet polynomial. Then
\[ \left(\sum_{n=1}^{N} |\mu(n)| |a_n|^2 [d(n)]^{\frac{\log q}{\log 2}-1}\right)^{1/2}\le \| F\|_q \]
whenever $0<q\le 2$.
\end{theorem}
To avoid unnecessary technicalities, we have chosen to state Theorem~\ref{embed} only for Dirichlet polynomials, but the result extends painlessly to the Hardy spaces of Dirichlet series $\Hp$ for $q\ge 1$, which were defined by Bayart in \cite{Ba} as the closure of the set of Dirichlet polynomials with respect to our $L^q$ norm. We refer to Chapter 6 of the recent book \cite{QQ} for additional information about these spaces, which have been studied by many authors in recent years.  Following Helson's argument, we will prove Theorem~\ref{embed} by first establishing the analogous inequality
\begin{equation} \label{circle} \left(|f(0)|^2+\frac{q}{2}|f'(0)|^2\right)^{1/2} \leq \norm{f}_{H^q(\T)}, \end{equation}
which will be shown to be valid for functions $f$ in $H^q(\T)$ for $0<q\le 2$. For the sake of completeness, we have included the case $q=2$ in the statement of the theorem, although it is trivial in view of the identity $\| F\|_2=(\sum_{n=1}^N |a_n|^2)^{1/2}$.

The proof of Theorem~\ref{embed} is presented in the next section, while the  proof of Theorem~\ref{thm1} is given in Section~\ref{proofzeta}. Section~\ref{ramanujan} contains the proof that (ii) is satisfied whenever \eqref{selberg} holds. The brief final Section~\ref{final} contains a few concluding remarks.

\section{Proof of Theorem~\ref{embed}}\label{proofembed}

We begin by recalling that $H^q(\T)$, $0<q<\infty$, consists of all functions $f(z)=\sum_{n=0}^{\infty} c_n z^n$ analytic in the unit disc $|z|<1$ such that
\[ \| f \|_{H^q(\T)}^q:= \sup_{r<1} \frac{1}{2\pi} \int_{0}^{2\pi} |f(re^{it}|^q dt <\infty. \]
Functions in $H^q(\T)$ have radial limits at almost every point of $\T$, and $H^q(\T)$ can alternatively be defined as a closed subspace of $L^q(\T)$; when $q\ge 1$, this is the subspace of functions $f$ whose Fourier coefficients $\hat{f}(k)$ vanish when $k$ is negative. For $q=2$,
we have just
\[ \| f\|_{H^2(\T)}^2=\sum_{n=0}^{\infty}  |c_n|^2. \]
For additional information about $H^q(\T)$, we refer to the standard references \cite{D, G}.

We now give a self contained and elementary proof of the basic estimate \eqref{circle}, which also can be obtained from a general inequality of Weissler \cite[Corollary 2.1]{W}:

\begin{lemma}\label{circlelem}
For $q \in (0,2]$ and arbitrary $f$ in $H^q(\T)$, we have
\begin{equation} \label{HL}
	 \left(|f(0)|^2+\frac{q}{2}|f'(0)|^2\right)^{1/2} \leq \norm{f}_{H^q(\T)}.
\end{equation}
\end{lemma}
\begin{proof}
Assume first that $f$ has no zeros in $\D$ and normalize $f$ so that $f(0)=1$. Then its Taylor series at $0$ has the form
\[ f(z)=1+az+ O(z^2). \]
Since $f$ has no zeros, $f^{q/2}$ can be defined to be analytic in $\D$ with Taylor series
\[ [f(z)]^{q/2}=1+\frac{q}{2}az+O(z^2). \]
It follows that
\[ \|f\|_{H^q(\T)}=\|f^{q/2}\|_{H^2(\T)}^{2/q}\ge \left[1+\left(\frac{q}{2}|a|\right)^2\right]^{1/q}\ge \left(1+\frac{q}{2}|a|^2\right)^{1/2},\]
where we in the final step used Bernoulli's inequality.

By a classical theorem of F. Riesz \cite[p. 53]{G}, a general function in $H^q(\T)$ has the form $Bf$, where $B$ is a Blaschke product and $f$ has no zeros in $\D$ \cite{G}. Moreover, $\|Bf\|_{H^q(\T)}=\|f\|_{H^q(\T)}$. It is clearly enough to consider only finite Blaschke products. To conclude, it therefore suffices for us to show that
\begin{equation} \label{toprove} |(bf)(0)|^2+\frac{q}{2}|(bf)'(0)|^2 \le |f(0)|^2+\frac{q}{2}|f'(0)|^2 \end{equation}
whenever $f$ is analytic in a neighborhood of $0$ and $b(z)=(z-w)/(1-\overline{w}z)$ is a general Blaschke factor with $|w|<1$. The case when $f(0)=0$ is trivial and hence we may assume as above that $f(0)=1$ and $f'(0)=a$. Then
\begin{eqnarray*} |(bf)(0)|^2+\frac{q}{2}|(bf)'(0)|^2 &=& |w|^2+\frac{q}{2}|1-|w|^2-wa|^2 \\
&\le& |w|^2+\frac{q}{2}(1-|w|^2)^2+q|a| |w|(1-|w|^2)+\frac{q}{2}|a|^2|w|^2 \\
&= &  1+ \frac{q}{2}|a|^2-(1-|w|^2)\left(1+\frac{q}{2}|a|^2-\frac{q}{2}(1-|w|^2)-q|a| |w|\right).
\end{eqnarray*}
Now \eqref{toprove} follows because
\[ 1+\frac{q}{2}|a|^2-\frac{q}{2}(1-|w|^2)-q|a| |w|\ge 1-\frac{q}{2}|w|^2-\frac{q}{2}(1-|w|^2)=1-\frac{q}{2}\ge 0.\]
\end{proof}

We now prepare for Helson's iterative argument by transforming our problem into a problem on $\T^{\pi(N)}$; here and in the sequel  $\pi(x)$ denotes the prime counting function. Let the prime factorization of the positive integer $1\le n\le N$ be
\[ n=\prod_{j=1}^{\pi(N)} p_j^{\alpha_j}, \]
where $p_1=2$, $p_2=3$, ... are the primes listed in ascending order. This means that we may represent $n$ by a unique multi-index
\[ \alpha(n):=(\alpha_1,...,\alpha_{\pi(N)}). \]
By the Bohr correspondence, our polynomial $F$ lifts to a polynomial
\[ \mathcal{B}F(z):=\sum_{n=1}^{N} a_n z^{\alpha(n)}\]
on $\T^{\pi(N)}$. It will be convenient to write just $\mathcal{B}F(z)=\sum_{\alpha}b_\alpha z^{\alpha} $, where it is understood that $b_{\alpha(n)}=a_n$.
Let $m$ denote normalized Lebesgue measure on $\T^d$, where $1\le d\le \pi(N)$. We will suppress the dependence on $d$ and instead write $dm(z_1,...,z_d)$ when there is a need to signify the dimension. We will use the fact that
\begin{equation} \label{ergodic}
\|F\|_q^q=\int_{\T^{\pi(N)}}\left|\mathcal{B} F(z)\right|^{q} dm(z).\end{equation}
This identity can be obtained from the Birkhoff--Khinchin ergodic theorem; an elementary proof can be found in \cite[Section 3]{SS}.

We will use Fubini's theorem and will then need the following version of Minkowski's inequality.

\begin{lemma}\label{minkowski} Let $X$ and $Y$ be measure spaces and $g$ a measurable
function on $X\times Y$. For $1\le r<\infty$,
\[ \big(\int_X\big(\int_Y |g(x,y)|dy\big)^rdx\big)^{1/r}
\le \int_Y\big(\int_X |g(x,y)|^r dx\big )^{1/r} dy.\]
\end{lemma}

\begin{proof}[Proof of Theorem~\ref{embed}]
For every $j=1,2,...$, we let $T_j$ be the operator defined on the set of polynomials by the following rule:
\[ T_j\left(\sum_{\alpha} b_{\alpha} z^{\alpha}\right):=\sum_{\alpha: \ \alpha_j\le 1} b_{\alpha} \left(\frac{q}{2}\right)^{\alpha_j/2} z^{\alpha}. \]
Hence $T_j f$ is linear in the variable $z_j$, and we obtain
\[ \mathcal{B}^{-1} T_{1}\cdots T_{\pi(N)} \mathcal B\left( \sum_{n\le N} a_n n^{-s}\right)=\sum_{n\le N} |\mu(n)| a_n [d(n)]^{\left(\frac{\log q}{\log 2}-1\right)/2} n^{-s}. \]
This means that we have
\[ S:=\left(\sum_{n=1}^{N} |\mu(n)| |a_n|^2 [d(n)]^{\frac{\log q}{\log 2}-1}\right)^{1/2}=
\left(\int_{\T_1}\cdots \int_{\T_{\pi(N)}} \left|T_{1}\cdots T_{\pi(N)} \mathcal BF(z)\right|^2 dm(z)\right)^{1/2}.\]
Using Fubini's theorem and applying Lemma~\ref{circlelem} in the variable $z_{1}$, we get
\[ S\le
\left(\int_{\T_2\times \cdots \times \T_{\pi(N)}} \left( \int_{\T_1} \left|T_2 \cdots T_{\pi(N)}\mathcal{B} F(z)\right|^{q} dm(z_1)\right)^{2/q} dm(z_2,...,z_{\pi(N)})
\right)^{1/2}.\]
In the next step we use Lemma~\ref{minkowski} with $r=2/q$ to get
\[ S\le
\left( \int_{\T_1}\left(\int_{\T_2\times \cdots \times \T_{\pi(N)}}\left|T_2 \cdots T_{\pi(N)}\mathcal{B} F(z)\right|^{2} dm(z_2,...,z_{\pi(N)})
 \right)^{1/2}  dm(z_1)\right)^{1/q}.\]
We now iterate this argument in each of the variables $z_2, ..., z_{\pi(N)}$. After $\pi(N)$ steps we reach the desired conclusion that
\[ \left(\sum_{n=1}^{N} |\mu(n)| |a_n|^2 [d(n)]^{\frac{\log q}{\log 2}-1}\right)^{1/2}\le
\left( \int_{\T^{\pi(N)}} |\mathcal{B} F(z)|^q dm(z)\right)^{1/q}=\|F\|_q, \]
where the last identity is \eqref{ergodic}.
\end{proof}

\section{Proof of Theorem~\ref{thm1}}\label{proofzeta}

Let $S_N$ denote the partial sum operator
\[ S_N\left(\sum_{n=1}^{\infty} a_n n^{-s} \right):=\sum_{n=1}^N a_n n^{-s}. \]
We consider it as an operator on $\Hp$ for $q\ge 1$, which we may define as  the closure of the set polynomials in the norm $\|\cdot\|_q$, as was done in \cite{Ba}. In \cite[Section 3]{AOS}, it is explained how the following lemma follows from a general result of Helson concerning compact Abelian groups whose dual is an ordered group \cite{H1}.  See also 8.7.2 and  8.7.6 of \cite{R}.

\begin{lemma}\label{partial}
For the partial sum operator $S_N$, we have the estimates
\[ \| S_N F\|_q\le \begin{cases} A_q \|F\|_q, &
1<q<\infty \\
 B_q \|F\|_1, & 0<q<1. \end{cases}\]
 for absolute constants $A_q$ and $B_q$.
\end{lemma}

The constants $A_q$ and $B_q$ are universal in the sense that they do not depend on the group in question; as in the classical M. Riesz theorem on $\T$, they are both of magnitude $|q-1|^{-1}$ when $q$ is close to $1$.

\begin{proof}[Proof of the bound from above in Theorem~\ref{thm1}]
We introduce the function
\[ F_N(s):=\prod_{p\le N} \sum_{m=0}^{\infty}a(p^m)p^{-m/2-ms}. \]
Since $a(n)$ is a multiplicative function, we have $D_N=S_N F_N$. In view of Lemma~\ref{partial}, it is therefore enough to estimate
$\| F_N\|_q$.

The function $F_N$ is clearly in $\Hp$ because its Dirichlet series is absolutely convergent for $\Real s=\sigma\ge 0$. For the same reason, \eqref{ergodic} remains valid, and we therefore find that
\[ \|F_N\|_q=\prod_{p\le N}\big\| \sum_{m=0}^{\infty} a(p^m) p^{-m/2}z^m \big\|_{H^q(\T)}. \]
By our assumption on $a(n)$,
\begin{equation} \label{largem}
\big\| \sum_{m=4}^{\infty} a(p^m) p^{-m/2}z^m \big\|_{H^q(\T)}\ll \sum_{m=4}^{\infty} p^{-m(1/2-\theta)}\ll p^{-(2-4\theta)}.
\end{equation}
On the other hand, for sufficiently large $p$, we have
\begin{eqnarray*}
\big\| \sum_{m=0}^{3} a(p^m) p^{-m/2}z^m \big\|_{H^q(\T)} &
= & \big\| \big( \sum_{m=0}^{3} a(p^m) p^{-m/2}z^m\big)^{q/2} \big\|_{H^2(\T)}^{2/q} \\
& \le & \left(1+\frac{q^2|a(p)|^2}{4p}+O(p^{-(2-4\theta)}) \right)^{1/q} \\ & = & 1+\frac{q|a(p)|^2}{4p}+O(p^{-(2-4\theta)}).
\end{eqnarray*}
Combining this estimate with \eqref{largem} and using that $\theta<1/4$, we obtain
\[ \| F_N\|_q \ll e^{q\lambda_{a}(N)/4}. \]
We can conclude immediately from Lemma~\ref{partial} when $q\neq 1$. Setting $q=1+1/\lambda_a(N)$ and recalling that $A_q$ is of magnitude $|q-1|^{-1}$, we also get from Lemma~\ref{partial}
that
\[  \| D_N\|_1 \le  \| D_N\|_q \ll \lambda_{a}(N) \| F_N\|_q \ll \lambda_{a}(N) e^{\lambda_{a}(N)/4}. \]
\end{proof}

For the proof of the bound from below in Theorem~\ref{thm1}, we require the following simple consequence of condition (B).

\begin{lemma}\label{square}
If a multiplicative arithmetic function $a(n)$ satisfies $\operatorname{(B)}$, then
\[ \sum_{p} \frac{|a(p)|^4}{p^2}<\infty.\]
\end{lemma}
\begin{proof}
We set $b(n):=|a(n)|^2/n.$ Then
\[ \prod_{p\le x/2} (1+ b(p))\ge \sum_{n\le x} |\mu(n)| b(n) - \sum_{x/2<p\le x} b(p). \]
In view of \eqref{normbelow}, this implies that there exists a positive constant $C$ such that
\[ \prod_{p\le x/2} (1+ b(p))\gg \prod_{p\le x} (1+ b(p)) - C \sum_{x/2<p\le x} b(p)\ge\sum_{x/2<p\le x}b(p) \cdot \left(\prod_{p\le x/2} (1+ b(p))-C\right). \]
It follows that $\sum_{x/2<p\le x}b(p)=O(1)$. Now our additional assumption from (B) that $a(p)\ll p^{\theta}$ gives the desired conclusion.
\end{proof}

We note that the relation $\sum_{x/2<p\le x}|a(p)|^2/p=O(1)$, obtained above as a consequence of \eqref{normbelow}, implies the growth condition \eqref{conseq}.

\begin{proof}[Proof of the bound from below in Theorem~\ref{thm1}]
In the range $0<q<2$, we use Theorem~\ref{embed} and set $r=\frac{\log q}{2}-1$ in condition (B). We then obtain
\[ \| D_N\|_q \gg \prod_{p\le N} \left(1+\frac{q}{2}|a(p)|^2p^{-1}\right)^{1/2}\ge \prod_{p\le N} \left(1+|a(p)|^2 p^{-1}\right)^{q/4}=e^{q\lambda_{a}(N)/4+O(1)},\]
where we in the last step used Lemma~\ref{square}.

To deal with the remaining case $q\ge 2$, we write
\[ [D_N(s)]^{k}=\sum_{n=1}^{N^k} a_{k,N}(n) n^{-1/2-s} \]
with
\[a_{k,N}(n):=\sum_{\substack{n_1\cdots n_k=n\\n_i\leq N}}a(n_1)\cdots a(n_k).\]
We pick $j\ge 1$ such that $2^j\le q<2^{j+1}$. We then apply  Lemma~\ref{embed} to $D_N^{2^j}$ and use that \[|\mu(n)| a_{2^j,N}(n)=
|\mu(n)| a(n) [d(n)]^j\] when $n\le N$ to obtain
\begin{eqnarray*} \| D_N\|_q^{2^j} = \| D_N^{2^j}\|_{q2^{-j}}& \ge & \left(\sum_{n=1}^{N} |\mu(n)| |a(n)|^2 n^{-1} [d(n)]^{2j+\frac{\log (q2^{-j})}{\log 2}-1}\right)^{1/2}\\
& = & \left(\sum_{n=1}^{N}
|\mu(n)||a(n)|^2 n^{-1} [d(n)]^{\frac{\log q}{\log 2}+j-1}\right)^{1/2}. \end{eqnarray*}
We now set $r=\frac{\log q}{2}+j-1$ in condition (B) and act as in the preceding case $0<q<2$.

\end{proof}

\section{The case $\lambda_a(x)=c\log\log x +O(1)$}\label{ramanujan}

We turn to the following positive result regarding our condition (B).
\begin{theorem}\label{rama}
Suppose $a(n)$ is a multiplicative arithmetic function satisfying
\begin{equation}\label{prime sum}  \lambda_{a}(x)=c \log\log x + O(1) \end{equation}
for some positive constant $c$.
Then part \eqref{normbelow} of condition $\operatorname{(B)}$ holds.
\end{theorem}

\begin{proof}
Let $r$ be real and consider the Dirichlet series
\begin{equation}F_r(s):=\sum_{n=1}^\infty\frac{|\mu(n)||a(n)|^2[d(n)]^r}{n^s}=\prod_p \left(1+2^r\frac{|a(p)|^2}{p^s}\right).
\end{equation}
Upon factoring out the zeta function we see that
\[F_r(s)=\zeta(s)^{c2^r}G_r(s),\]
where
\[G_r(s):=\prod_p \bigg(1-\frac{1}{p^s}\bigg)^{c2^r} \left(1+2^r\frac{|a(p)|^2}{p^s}\right)=\prod_p \left(1+2^r\frac{|a(p)|^2-c}{p^{s}}+O\big(p^{-2\sigma}\big)\right).\]
It follows from \eqref{prime sum} that $G_r(s)$ is analytic in a neighborhood of $\sigma=1$. The usual methods (e.g. Theorem 2 of \cite{S}) applied to $F_r(s+1)$ now give
\begin{equation}\sum_{n\leq x}|\mu(n)||a(n)|^2[d(n)]^rn^{-1}=\frac{G_r(1)}{\Gamma(c2^r+1)}[\log x]^{c2^r}+O\big([\log x]^{c2^r-1}\big).
\end{equation}
On the other hand,
\[\prod_{p\leq x}\left(1+2^r\frac{|a(p)|^2}{p}\right)\leq \exp\Big(2^r\lambda_\alpha(x)\Big)\ll [\log x]^{c2^r},\]
and the so the result follows.
\end{proof}

It is of interest to note that, under the assumption \eqref{prime sum}, we can determine the asymptotic behavior of
\[G_r(1)=\prod_p \bigg(1-\frac{1}{p}\bigg)^{c2^r} \left(1+2^r\frac{|a(p)|^2}{p}\right)\]
when $r$ is large.
Indeed,
\begin{equation}\label{est 1}
\begin{split}\prod_{p>c2^r} \bigg(1-\frac{1}{p}\bigg)^{c2^r} \left(1+2^r\frac{|a(p)|^2}{p}\right)=&\prod_{p>c2^r} \left(1+2^r\frac{|a(p)|^2-c}{p}+O\big(|a(p)|^2(2^r/p)^2\big)\right)\\
\ll&\exp\left(\sum_{p>c2^r}\bigg[2^r\frac{|a(p)|^2-c}{p}+O\big(|a(p)|^2(2^r/p)^2\big)\bigg]\right)
\\=&\exp(o(2^r)),
\end{split}
\end{equation}
where  we have used \eqref{prime sum} and partial summation for the sum of the `big O' term. By Mertens' third theorem, we have
\begin{equation}\label{est 2}\prod_{p\leq c2^r}\bigg(1-\frac{1}{p}\bigg)^{c2^r}\sim\left(\frac{e^{-\gamma}}{\log(c2^r)}\right)^{c2^r}.
\end{equation}
For the final product we first note that 
\[\sum_{p\leq x}\frac{|a(p)|}{p^{1/2}}\ll  \sqrt{\pi(x)\lambda_a(x)}\ll  \sqrt{\frac{x\log\log x}{\log x}}\]
by the Cauchy--Schwarz inequality along with the prime number theorem and  \eqref{prime sum}. 
This gives
\begin{equation}\label{est 3}
\begin{split}\prod_{p\leq c2^r}\left(1+2^r\frac{|a(p)|^2}{p}\right)\leq \prod_{p\leq c2^r}\left(1+2^{r/2}\frac{|a(p)|}{p^{1/2}}\right)^2\leq \exp\left(2\cdot 2^{r/2}\sum_{p\leq c2^r}\frac{|a(p)|}{p^{1/2}}\right)=\exp\big(o(2^r)\big).
\end{split}
\end{equation}
On combining \eqref{est 1}, \eqref{est 2}, and \eqref{est 3} we see that
\begin{equation}\label{exp decay}G_r(1)=\exp\big[-c2^r\big(\log r+O(1)\big)\big].
\end{equation}

\section{Final remarks}\label{final}

By keeping track of the constant in our upper bound for $\|D_N\|_q $, we see that it grows super-exponentially with $q$. However, from \eqref{exp decay} we see that the constant in our lower bound is of super-exponential decay.  We believe that the latter behavior is the true order of growth. This conjecture is supported by the result of Conrey and Gamburd \cite{CG} in the distinguished case $a(n)\equiv 1$, stating that for $k\in\mathbb{N}$,
\[\lim_{T\to\infty}\frac{1}{T}\int_0^T \Big|\sum_{n\leq N}n^{-1/2-it}\Big|^{2k}\sim \alpha_k\gamma_k(\log N)^{k^2},\]
where $\alpha_k$ is an arithmetic factor similar to $G_r(1)$, and $\gamma_k$ is the volume of a particular convex polytope. Since the latter quantity is at least bounded, we see that these constants share the same behavior as those in our lower bound.

The picture changes when $q\to 1$. Then the constant in Helson's version of the M. Riesz theorem is the one that leads to the blow-up of our estimate. It remains an interesting problem to determine the precise order of growth in the range $0< q \le 1$.

As we noted earlier, our results are in line with the conjectures for moments of  primitive $L$-functions from the Selberg class. Under certain orthogonality conditions on the coefficients (e.g.  (1.13) of \cite{S2}), our methods should extend to products of Dirichlet polynomials. We expect the resultant bounds on the norm to remain consistent with the analogous conjectures for moments of non-primitive $L$-functions \cite{He,Mil}.

Finally, we close the paper with some additional remarks pertaining to Theorem~\ref{embed}. 
A natural question is whether the M\"{o}bius function is really needed in our inequality when $q\neq 1$. For sufficiently small $q$, this is indeed so, as can be seen from the size of the Taylor coefficients of the function $(1-z)^{-1/(2q)}$.  In the range $1<q<2$, we do not know, but here it is of interest to note that a standard interpolation argument gives the inequality
\[ \left(\sum_{n=0}^{\infty} |c_n|^2 (n+1)^{1-2/q}\right)^{1/2}\le C_q \| f\|_{H^q(\T)} \]
for some constant $C_q$. However, since $1-2/q>\frac{\log q}{\log 2}-1$ when $1<q<2$ and the exponent in Theorem~\ref{embed} can not be improved, it is clear that\footnote{In \cite[p. 2693]{O}, it is claimed that $C_q=1$ for $1<q<2$, but this is unfortunately not correct.} $C_q>1$.

The problem raised in the preceding paragraph is to find the largest exponent $\gamma=\gamma(p)$ for which the contractive inequality
\[ \left(\sum_{n=0}^{\infty} |c_n|^2 (n+1)^{\gamma}\right)^{1/2}\le \| f\|_{H^q(\T)}\]
holds for every $f$ in $H^q(\T)$, $0<q<2$. We note that $\gamma(q)$ exists for every $0<q<2$ because a result of Burbea \cite{B} implies that
$\gamma(2/\ell)\ge 1-\ell$ for $\ell=2,3,...$.  But this result shows also that there is a considerable gap between the known upper and lower bounds for $\gamma(q)$, and it remains an interesting problem to estimate this quantity more precisely for $0<q<2$, $q\neq 1$.

A similar question appears for $q>2$ because it is known from \cite[Lemma 8]{Se} that
\eqref{helson} is reversed when $q=2^j$ for a positive integer $j$:
\begin{equation}\label{upper} \left(\sum_{n=1}^{N} |a_n|^2 [d(n)]^{j-1}\right)^{1/2}\ge \| F\|_{2^j}. \end{equation}
In \cite{Se}, a variant of the Riesz--Thorin interpolation method was used to obtain a similar inequality in the range $2<q<4$, but with a power of $d(n)$ larger than $\frac{\log q}{\log 2}-1$. It seems reasonable to conjecture that \eqref{upper} should hold whenever $j$ is a real number larger than $1$. If such a result could be established, we could use it to obtain the bound in \eqref{above} for $q>2$ and thus obtain the conjectured behavior of the implied constant when $q\to \infty$.

\section*{Acknowledgement} The authors are grateful to Herv\'{e} Queff\'{e}lec for drawing their attention to Weissler's paper \cite{W} and also to Titus Hilberdink, Chris Hughes, Eero Saksman, and the anonymous referee for helpful comments.

\end{document}